\documentclass[12pt]{amsart}
\usepackage{dirtytalk,amsmath,amsfonts,amssymb,amscd,mathrsfs,graphicx,pst-node,tikz-cd,fancyvrb,lipsum,amsthm,graphicx,epsf,amsmath,verbatim}
\usepackage[colorlinks=true,linkcolor=blue,citecolor=blue]{hyperref}
\usepackage{times}
\usepackage{enumerate}

\def\line#1{\hbox to \hsize{#1\hfill}}
\newtheorem{prop}{Proposition}

\newtheorem{defi}{Definition}
\newtheorem{lemm}{Lemma}
\newtheorem{theo}{Theorem}
\newtheorem{maintheo}{Main Theorem}
\newtheorem{coro}{Corollary}
\newtheorem{exem}{Example}

\newcommand{\<}{\langle}
\renewcommand{\>}{\rangle}

\title{Almost Paracomplex Structures on 4-Manifolds}

\author{Nikos Georgiou}
\address{Nikos Georgiou\\
  Department of Mathematics\\
          Waterford Institute of Technology\\
          Waterford\\
          Co. Waterford\\
          Ireland.}
\email{ngeorgiou@wit.ie}

\author{Brendan Guilfoyle}
\address{Brendan Guilfoyle\\
          School of STEM\\
          Munster Technological University, Kerry\\
          Tralee\\
          Co. Kerry\\
          Ireland.}
\email{brendan.guilfoyle@mtu.ie}
\begin{document} 

\maketitle
\begin{abstract}
 
    Reflection in a line in Euclidean 3-space defines an almost paracomplex structure on the space of all oriented lines, isometric with respect to the canonical neutral Kaehler metric. Beyond Euclidean 3-space, the space of oriented geodesics of any real 3-dimensional space form admits both isometric and anti-isometric paracomplex structures.
   
    This paper considers the existence or otherwise of isometric and anti-isometric almost paracomplex structures $j$ on a pseudo-Riemannian 4-manifold $(M,g)$, such that $j$ is parallel with respect to the Levi-Civita connection of $g$. It is shown that if an isometric or anti-isometric almost paracomplex structure on a conformally flat manifold is parallel, then the scalar curvature of the metric must be zero. In addition, it is found that $j$ is parallel iff the eigenplanes are tangent to a pair of mutually orthogonal foliations by totally geodesic surfaces.
   
   The composition of a Riemannian metric with an isometric almost paracomplex structure $j$ yields a neutral metric $g'$.  It is proven that if $j$ is parallel, then $g$ is Einstein iff $g'$ is conformally flat and scalar flat.
   
    The vanishing of the Hirzebruch signature is found to be a necessary topological condition for a closed 4-manifold to admit an Einstein metric with a parallel isometric paracomplex structure. Thus, while the K3 manifold admits an Einstein metric with an isometric paracomplex structure, it cannot be parallel. The same holds true for certain connected sums of complex projective 2-space and its conjugate.
    
\end{abstract}

\tableofcontents

\section{Introduction and Results}\label{s:0}

The purpose of this paper is to explore isometric and anti-isometric almost paracomplex structures on conformally flat 4-manifolds and to see to what extent the examples of such structures on spaces of oriented geodesics and product manifolds are typical. 

In general, an almost paracomplex structure on a smooth 2$n$-dimensional manifold $M$ is an endomorphism of its tangent space at each point, one that squares to the identity (but is not the identity) and has a pair of $n$-dimensional eigenspaces with eigenvalues $\pm 1$. Conversely, the choice of a pair of transverse $n$-planes at each point on $M$ determines an almost paracomplex structure, unique up to an overall sign.

A pointwise choice of complimentary subspaces on a manifold is called an {\em almost product structure} \cite{naveira} \cite{yano} and clearly the product of two manifolds carries such a structure. However, in 4-dimensions, the analogy with almost complex structures means that a 2+2 almost product structure is generally referred to as an almost paracomplex structure. 

As with their better-known cousins, almost paracomplex structures are often associated with metrics, although the signature of the metric may be indefinite. Our approach is to consider paracomplex structures through their eigenplane distributions, as has been done in the complex case \cite{Pat}. Note that a parallel isometric (para)complex structure is often referred to as a (para)K\"ahler structure.

We start in Section \ref{s:1} by considering the reflection of an oriented line in ${\mathbb R}^3$ in a fixed line. By applying this to Jacobi fields along a fixed oriented line (infinitesimal variations of the line), one obtains a canonical almost paracomplex structure $J_1$ on the 4-manifold ${\mathbb L}({\mathbb R}^3)$ of all oriented lines.

The space ${\mathbb L}({\mathbb R}^3)$ admits a canonical neutral K\"ahler structure \cite{kahlermetric} $({\mathbb G},J_0,\Omega_0)$ consisting of a metric of neutral signature $(2,2)$, a complex structure and a symplectic 2-form, respectively, satisfying the conditions
\[
{\mathbb G}(J_0\cdot,J_0\cdot)={\mathbb G}(\cdot,\cdot) \qquad \qquad\qquad
{\mathbb G}(J_0\cdot,\cdot)=\Omega_0(\cdot,\cdot).
\]
If the first of these conditions holds, we say that $J_0$ is {\em isometric} with respect to ${\mathbb G}$. The almost paracomplex structure $J_1$ induced by reflection in a line is 
 {\em anti-isometric}: ${\mathbb G}(J_1\cdot,J_1\cdot)=-{\mathbb G}(\cdot,\cdot)$.

Our first result in Section \ref{s:1} is:

\vspace{0.1in}
\begin{maintheo}\label{t:1}
The space of oriented lines of Euclidean 3-space admits a commuting triple $(J_0,J_1,J_2)$ of a complex structure, an almost paracomplex structure and an almost complex structure, respectively, satisfying $J_2=J_0J_1$.
The complex structure $J_0$ is isometric, while $J_1$ and $J_2$ are anti-isometric. Only $J_0$ is parallel w.r.t. ${\mathbb G}$, and only $J_0$ is integrable. 

Composing the neutral metric ${\mathbb G}$ with the (para)complex structures $J_0,J_1,J_2$ yields closed 2-forms $\Omega_0$ and $\Omega_1$, and a conformally flat, scalar flat, neutral metric  $\tilde{\mathbb G}$, respectively. 

The neutral 4-manifolds $({\mathbb L}({\mathbb R}^3),{\mathbb G})$ and $({\mathbb L}({\mathbb R}^3),\tilde{\mathbb G})$ are isometric. Only $J_0$ is parallel w.r.t. $\tilde{\mathbb G}$.
\end{maintheo}
\vspace{0.1in}

Section \ref{s:5} reconsiders oriented geodesics, this time in general 3-dimensional spaces of constant curvature. Here we find examples of non-flat parallel isometric and anti-isometric paracomplex structures. In particular, it is shown that the Hodge star operator, through the Pl\"ucker embedding, gives rise to a natural (para)complex structure on these geodesic spaces. Section \ref{s:5} also discusses other non-trivial examples that arise from the product of two surfaces.

In Sections \ref{s:2} and  \ref{s:3} the following geometric characterizations are given: an almost paracomplex structure on a pseudo-Riemannian 4-manifold is isometric iff its eigenplanes are orthogonal, while it is anti-isometric iff its eigenplanes are totally null. Thus, there are no anti-isometric paracomplex structures when the metric is definite, but isometric paracomplex structures can exist for both definite and indefinite metrics. 

The characterisation can be suitably generalised to higher dimensions, but this paper will stay exclusively in dimension 4.
 
In Sections \ref{s:2} and  \ref{s:3} it is proven that:

\vspace{0.1in}
\begin{maintheo}\label{t:2}
A  conformally flat neutral metric on a 4-manifold that admits a parallel anti-isometric or isometric almost paracomplex structure has zero scalar curvature.
\end{maintheo}
\vspace{0.1in}
The proof shows that the vanishing of the scalar curvature arises as a consistency condition for parallel anti-isometric or isometric almost paracomplex structures. In the neutral case, this vanishing is equivalent to the conformal factor of the metric satisfying the ultrahyperbolic equation \cite{CG}.

Section \ref{s:4} reformulates the parallel condition for an isometric almost paracomplex structure in terms of the first order invariants of the eigenplane distributions:
\vspace{0.1in}
\begin{maintheo}\label{t:3}
Let $j$ be an isometric almost paracomplex structure on a pseudo-Riemannian 4-manifold. Then $j$ is parallel iff the eigenplane distributions are tangent to a pair of mutually orthogonal foliations by totally geodesic surfaces.
\end{maintheo}
\vspace{0.1in}
In particular, the eigenplane distributions are integrable in the sense of Frobenius. Since in common parlance the integrability of the eigen-distribution allows one to drop the term {\em almost}, we can say {\em parallel paracomplex structure} rather than {\em parallel almost paracomplex structure.}

Section \ref{s:6} establishes a duality between Riemannian Einstein 4-manifolds and conformally flat neutral 4-manifolds in the presence of a parallel isometric paracomplex structure.  We prove:

\vspace{0.1in}
\begin{maintheo}\label{t:4}
Let $(M,g)$ be a Riemannian $4$-manifold endowed with a parallel isometric paracomplex structure $j$, and let the associated neutral metric be $g'(\cdot,\cdot)=g(j\cdot,\cdot)$. Then, $g'$ is locally conformally flat if and only if $g$ is Einstein.
\end{maintheo}
\vspace{0.1in}
Thus given a Riemannian Einstein metric $g$ on a 4-manifold $M$ admitting a parallel isometric paracomplex structure $j$, one can construct a conformally flat, scalar flat neutral metric $g'$, and vice versa.

In the final section this transformation is applied to the Chern-Gauss-Bonnet formula for the Euler characteristic $\chi(M)$, together with the integral expression for the Hirzebruch signature $\tau(M)$. We obtain global obstructions to parallel isometric paracomplex structures on closed conformally flat neutral 4-manifolds:

\vspace{0.1in}
\begin{maintheo}\label{t:5}
Let $(M,g')$ be a closed, conformally flat, scalar flat, neutral 4-manifold. If $g'$ admits a parallel isometric paracomplex structure, then 
\[
\tau(M)=0\qquad\qquad{\mbox{ and }}\qquad\qquad\chi(M)\geq 0.
\]

If, moreover, the Ricci tensor of $g'$ has negative norm $|Ric(g')|^2\leq 0$, then $M$ admits a flat Riemannian metric.
\end{maintheo}
\vspace{0.1in}

Using the opposite construction, global obstructions to parallel isometric paracomplex structures on closed Einstein 4-manifolds emerge:

\vspace{0.1in}
\begin{maintheo}\label{t:6}
Let $(M,g)$ be a closed Riemannian Einstein 4-manifold. 

If $g$ admits a parallel isometric paracomplex structure, then $\tau(M)=0$. 
\end{maintheo}
\vspace{0.1in}
As a corollary, the $K3$ 4-manifold, as well as the 4-manifolds ${\mathbb{ C}}P^2\#k\overline{\mathbb CP}^2$ for $k=3,5,7$, admit Riemannian Einstein metrics and isometric almost paracomplex structures, but these almost paracomplex structures cannot be parallel.

\vspace{0.1in}

\section{Almost Paracomplex Structure on Oriented Line Space}\label{s:1}
In this section the canonical almost paracomplex structure on the space of oriented lines of Euclidean 3-space is defined.

\subsection{Oriented line space}

The space ${\mathbb L}({\mathbb R}^3)$ of oriented lines in ${\mathbb R}^3$ is well known to be modelled on the total space  $TS^2$ of the tangent bundle of the 2-sphere \cite{kahlermetric}. 

Taking the usual holomorphic coordinate $\xi$ on $S^2$ obtained by stereographic projection from the south pole, one can construct holomorphic coordinates $(\xi,\eta)$ on ${\mathbb L}({\mathbb R}^3)$ by identifying 
\[
(\xi,\eta)\leftrightarrow \eta\frac{\partial}{\partial\xi}+\bar{\eta}\frac{\partial}{\partial\bar{\xi}}\in T_\xi S^2.
\]
Thus $\xi\in S^2$ is the direction of the oriented line and the complex number $\eta$ represents the orthogonal displacement of the line from the origin.

\vspace{0.1in}
\begin{theo}\cite{kahlermetric}\label{t:gk}
The space ${\mathbb L}({\mathbb R}^3)$ of oriented lines of ${\mathbb R}^3$ admits a canonical metric ${\mathbb G}$ that is invariant under the Euclidean group acting on lines. The metric is of neutral signature (2,2), is conformally flat and scalar flat, but not Einstein. 

It can be supplemented by a complex structure $ J_0$ and symplectic structure $\omega$, so that $({\mathbb L}({\mathbb R}^3),{\mathbb G}, J_0,\omega)$ is a neutral K\"ahler 4-manifold.
\end{theo}
\vspace{0.1in}

In terms of the local coordinates $(\xi,\eta)$ the neutral metric is
\begin{equation}\label{e:metric}
ds^2=4(1+\xi\bar{\xi})^{-2}{\mathbb{I}}\mbox{m}\left(d\bar{\eta} d\xi+\frac{2\bar{\xi}\eta}{1+\xi\bar{\xi}}d\xi d\bar{\xi}\right),
\end{equation}
and the action of the complex structure is:
\begin{equation}\label{e:j0}
 J_0\left(\frac{\partial}{\partial\xi}\right)=i\frac{\partial}{\partial\xi}
\qquad\qquad
 J_0\left(\frac{\partial}{\partial\eta}\right)=i\frac{\partial}{\partial\eta}.
\end{equation}
The complex structure can be defined as the rotation of Jacobi fields along a fixed oriented line through 90$^o$, as follows. 

Given a fixed oriented line $\gamma_0$, the tangent space to the space of oriented lines at $\gamma_0$ can be identified with the vector fields along the line that are orthogonal to it and satisfy the Jacobi equation: 
\[
\frac{D^2}{dr^2}V=0,
\]
$r$ being an affine parameterisation of the oriented line. Thus the solutions can be seen as infinitesimal nearby lines with a linear dependence
\[
V=V_0+rV_1,
\]
$V_0$ and $V_1$ being constant orthogonal vectors along $\gamma_0$,

The link between ${\mathbb L}({\mathbb R}^3)$ and ${\mathbb R}^3$ can be made by the map $\Phi:{\mathbb L}({\mathbb R}^3)\times{\mathbb R}\rightarrow{\mathbb R}^3$ which takes an oriented line and a number, to the point in ${\mathbb R}^3$ on the line that lies the given distance from the point on the line closest to the origin. In coordinates $(\xi,\eta)$ on ${\mathbb L}({\mathbb R}^3)$ and $(x_1,x_2,x_3)$ on ${\mathbb R}^3$, $\Phi((\xi,\eta),r)$ can be written 
\begin{equation}\label{e:minit}
z=x_1+ix_2=\frac{2(\eta-\xi^2\bar{\eta})}{(1+\xi\bar{\xi})^2}+\frac{2\xi}{1+\xi\bar{\xi}}r
\qquad
x_3=\frac{-2(\bar{\xi}\eta+\xi\bar{\eta})}{(1+\xi\bar{\xi})^2}+\frac{1-\xi\bar{\xi}}{1+\xi\bar{\xi}}r.
\end{equation}
To explicitly see the Jacobi fields along the fixed oriented line $\gamma_0$, introduce the follow null frame along $\gamma_0$
\[
e_0=\frac{2\xi_0}{1+\xi_0\bar{\xi}_0}\frac{\partial}{\partial z}+\frac{2\bar{\xi}_0}{1+\xi_0\bar{\xi}_0}\frac{\partial}{\partial \bar{z}}+\frac{1-\xi_0\bar{\xi}_0}{1+\xi_0\bar{\xi}_0}\frac{\partial}{\partial x_3}
\]
\[
e_+=\frac{\sqrt{2}}{1+\xi_0\bar{\xi}_0}\frac{\partial}{\partial z}-\frac{\sqrt{2}\bar{\xi}_0^2}{1+\xi_0\bar{\xi}_0}\frac{\partial}{\partial \bar{z}}-\frac{\sqrt{2}\bar{\xi}_0}{1+\xi_0\bar{\xi}_0}\frac{\partial}{\partial x_3}
\]
\[
e_-=-\frac{\sqrt{2}{\xi}_0^2}{1+\xi_0\bar{\xi}_0}\frac{\partial}{\partial z}+\frac{\sqrt{2}}{1+\xi_0\bar{\xi}_0}\frac{\partial}{\partial \bar{z}}-\frac{\sqrt{2}{\xi}_0}{1+\xi_0\bar{\xi}_0}\frac{\partial}{\partial x_3},
\]
where $e_0$ is a unit vector along pointing along $\gamma_0$ and the unit real vectors $e_1, e_2$ defined by $e_\pm=e_1\pm i e_2$ are mutually orthogonal and orthogonal to $\gamma_0$.

Now consider the derivative $D\Phi:T_{((\xi,\eta),r)}{\mathbb L}({\mathbb R}^3)\rightarrow T_{(x_1,x_2,x_3)}{\mathbb R}^3\rightarrow $, it can be easily compute that
\[
D\Phi\left|_{((\xi,\eta),r)}\right.\left(\frac{\partial}{\partial\xi}\right)=\left(r-\frac{2\bar{\xi}\eta}{1+\xi\bar{\xi}}\right)\frac{\sqrt{2}}{1+\xi\bar{\xi}}\;e_+-\frac{2\bar{\eta}}{(1+\xi\bar{\xi})^2}\;e_0
\]
\[
D\Phi\left|_{((\xi,\eta),r)}\right.\left(\frac{\partial}{\partial\eta}\right)=\frac{\sqrt{2}}{1+\xi\bar{\xi}}\;e_+.
\]
Projecting perpendicular to $\gamma_0$ gives the identification of the tangent space to ${\mathbb L}({\mathbb R}^3)$ at $\gamma_0$ with the orthogonal Jacobi fields along the line. Now the rotation of Jacobi fields about the line, namely 
\[
j_0(e_+)=ie_+ \qquad j_0(e_-)=-ie_-,
\]
clearly induces the almost complex structure $J_0$ in equation (\ref{e:j0}). 

The almost complex structure $J_0$ is {\em integrable} in that the Nijinhuis tensor of $J_0$ vanishes - a fact first noted in the modern literature by Hitchin \cite{hitch0}, but dating back to Weierstrass \cite{Weier}. Here the Nijinhuis tensor of an almost (para)complex structure $J$ is defined
\[
N_{ij}^k=J_i^l\partial_lJ_j^k-J_j^l\partial_lJ_i^k-J_l^k(\partial_iJ_j^l-\partial_jJ_i^l),
\]
and when it vanishes $J$ is said to be a {\it (para)complex structure}. 

Moreover,  $J_0$ is {\em parallel} with respect to the Levi-Civita connection of ${\mathbb G}$:
\[
\nabla_{\mathbb G} J_0=0.
\]
In the case where ${\mathbb G}$ is Riemannian and $J_0$ an isometric complex structure, this condition is referred to as the K\"ahler condition. When ${\mathbb G}$ is pseudo-Riemannian and $J_0$ an anti-isometric paracomplex structure, it is referred to as the para-K\"ahler condition. 

\vspace{0.1in}
\subsection{Reflection in a line}

Consider then the {\em reflection} of an oriented line $\gamma$ in a fixed line $\gamma_0$ (the orientation of $\gamma_0$ is immaterial). This gives rise to a new oriented line $\gamma'$ and can be viewed as a mapping ${\mathcal R}_{\gamma_0}:{\mathbb L}({\mathbb R}^3)\rightarrow {\mathbb L}({\mathbb R}^3)$.

\vspace{0.1in}
\begin{lemm}\label{l:1}
In the special case where $\gamma_0$ is the $x_3-$axis, the map in holomorphic coordinates is simply
\[
{\mathcal R}_0(\xi,\eta)=\left(\frac{1}{\bar{\xi}},\frac{\bar{\eta}}{\bar{\xi}^2}\right).
\]
\end{lemm}
\begin{proof}
This can be easily checked by substitution of the transformation $((\xi,\eta),r)\rightarrow({\mathcal R}_0(\xi,\eta),-r)$ into equations (\ref{e:minit}) and noting that the induced transformation is $(x_1,x_2,x_3)\rightarrow(-x_1,-x_2,x_3)$, as claimed.
\end{proof}
\vspace{0.1in}
Given a Jacobi field along a fixed oriented line $\gamma_0$, reflection in $\gamma_0$ yields another Jacobi field and hence yields an endomorphism $J_1$ of the tangent space to ${\mathbb L}({\mathbb R}^3)$ at $\gamma_0$. The map $J_1$ squares to the identity, although it is not equal to the identity, and it has two 2-dimensional eigenspaces. Thus $J_1$ is an {\em almost paracomplex structure} on ${\mathbb L}({\mathbb R}^3)$. 

\vspace{0.1in}
\begin{prop}
Reflection defines the almost paracomplex structure 
\[
J_1\left(\frac{\partial}{\partial\xi}\right)= \frac{\partial}{\partial\xi}+\frac{4\bar{\xi}\eta}{1+\xi\bar{\xi}}\frac{\partial}{\partial\eta} \qquad\qquad
J_1\left(\frac{\partial}{\partial\eta}\right)= -\frac{\partial}{\partial\eta}.
\]
This is anti-isometric: ${\mathbb G}(J_1\cdot,J_1\cdot)=-{\mathbb G}(\cdot,\cdot)$, but is neither parallel  with respect to  ${\mathbb G}$ nor integrable.
\end{prop}
\begin{proof}
Adopting the earlier identification of Jacobi fields, the paracomplex action is
\[
j_1(e_+)=-e_+ \qquad j_1(e_-)=-e_-
\]
and so
\begin{eqnarray}
D\Phi\;\left|_{(\xi,\eta)}\right.J_1\left(\frac{\partial}{\partial\xi}\right)&=-\left(-r-\frac{2\bar{\xi}\eta}{1+\xi\bar{\xi}}\right)\frac{\sqrt{2}}{1+\xi\bar{\xi}}\;e_+-\frac{2\bar{\eta}}{(1+\xi\bar{\xi})^2}\;e_0\nonumber\\
&=D\Phi\;\left|_{(\xi,\eta)}\right.\left(\frac{\partial}{\partial\xi}\right)+\frac{4\sqrt{2}\bar{\xi}\eta}{(1+\xi\bar{\xi})^2}e_+\nonumber
\end{eqnarray}

\[
D\Phi\left|_{(\xi,\eta)}\right.J_1\left(\frac{\partial}{\partial\eta}\right)=-\frac{\sqrt{2}}{1+\xi\bar{\xi}}\;e_+.
\]
The last equation implies that 
\[
J_1\left(\frac{\partial}{\partial\eta}\right)= -\frac{\partial}{\partial\eta},
\]
and the first equation of the Proposition follows.

A straightforward computation shows that $J_1$ is neither parallel nor integrable.

\end{proof}
\vspace{0.1in}

The canonical complex structure $J_0$ on ${\mathbb L}({\mathbb R}^3)$ combines with the almost paracomplex structure $J_1$ to form $J_2=J_0J_1$. 

\vspace{0.1in}
\noindent{\bf Proof of Main Theorem} \ref{t:1}.

By definition, $J_2$ is an almost complex structure
\[
J_2\left(\frac{\partial}{\partial\xi}\right)= i\frac{\partial}{\partial\xi}+\frac{4i\bar{\xi}\eta}{1+\xi\bar{\xi}}\frac{\partial}{\partial\eta} \qquad\qquad
J_2\left(\frac{\partial}{\partial\eta}\right)= -i\frac{\partial}{\partial\eta},
\]
and it is easily checked that it commutes with $J_0$ and $J_1$.

From the local expression (\ref{e:metric}) of the metric and of the (para)complex structures one computes that
\[
{\mathbb G}(J_0\cdot,J_0\cdot)={\mathbb G}(\cdot,\cdot) \qquad\qquad {\mathbb G}(J_1\cdot,J_1\cdot)=-{\mathbb G}(\cdot,\cdot) \qquad\qquad {\mathbb G}(J_2\cdot,J_2\cdot)=-{\mathbb G}(\cdot,\cdot),
\]
as claimed. 

The integrability of $J_0$ is well-known \cite{hitch0} and after a short computation the Nijinhuis tensors for $J_1$ and $J_2$ are found not to vanish. Similarly, only the covariant derivative of $J_0$ vanishes identically.

Define
\[
\Omega_0(X,Y)={\mathbb G}(J_0X,Y) \qquad \Omega_1(X,Y)={\mathbb G}(J_1X,Y) \qquad
\tilde{\mathbb G}(X,Y)={\mathbb G}(J_2X,Y). 
\]
The metrics ${\mathbb G}$ and $\tilde{\mathbb G}$ are easily seen to be isometric via the map $(\xi,\eta)\mapsto(\xi,i\eta)$.
 
\qed

\vspace{0.1in}

\subsection{Conformal representation}

A pseudo-Riemannian manifold $(M,g)$ is {\em conformally flat} if it can be covered by local coordinate systems $(X_1,X_2,...,X_n)$ such that the metric is 
\[
ds^2=\Omega^2\sum_{k=1}^n \pm dX^2_k,
\]
where $\Omega$ is a non-zero function defined on the coordinate neighbourhood. Equivalently, the Weyl curvature tensor of the metric vanishes at every point.

The neutral metric ${\mathbb G}$ on the space of oriented lines is conformally flat and therefore such a local conformal coordinate system and conformal factor must exist. To find them, recall that a line in ${\mathbb R}^3$ can be uniquely determined by two distinct points $(s_1,s_2,s_3)$ and $(t_1,t_2,t_3)$ on the line. Pl\"ucker coordinates form the sextet $(p_1,p_2,p_3,q_1,q_2,q_3)$ defined by \cite{fjohn}
\[
p_1=s_2t_3-t_2s_3 \qquad p_2=s_3t_1-t_3s_1 \qquad p_3=s_1t_2-t_1s_2
\]
\[
q_1=s_1-t_1 \qquad q_2=s_2-t_2 \qquad q_3=s_3-t_3.
\]
Conformal coordinates on line space are then given by
\[
X_1=\frac{p_2+q_2}{q_3} \quad X_2=\frac{-p_1-q_1}{q_3} \quad X_3=\frac{p_2-q_2}{q_3} \quad X_4=\frac{-p_1+q_1}{q_3}.
\]
These originate in the representation of lines in ${\mathbb R}^3$ in terms of the Grassmannian of oriented 2-planes in ${\mathbb R}^4$.

As a  consequence:
\vspace{0.1in}
\begin{prop}\label{p:conf}\cite{CG}
For complex coordinates $(\xi,\eta)$ on ${\mathbb L}({\mathbb R}^3)$, over the upper hemisphere $|\xi|^2<1$ the conformal coordinates $(Z_1=X_1+iX_2,Z_2=X_3+iX_4)$ are
\begin{equation}\label{e:confco1}
Z_1=\frac{2}{1-\xi^2\bar{\xi}^2}\left(\eta+\xi^2\bar{\eta}-i(1+\xi\bar{\xi})\xi\right),
\end{equation}
\begin{equation}\label{e:confco2}
Z_2=\frac{2}{1-\xi^2\bar{\xi}^2}\left(\eta+\xi^2\bar{\eta}+i(1+\xi\bar{\xi})\xi\right),
\end{equation}
with inverse
\begin{equation}\label{e:confcoinv1}
    \xi=\frac{i(Z_1-Z_2)}{2-\sqrt{4+|Z_1-Z_2|^2}},
\end{equation}
\begin{equation}\label{e:confcoinv2}
    \eta=\frac{Z_1+Z_2}{2-\sqrt{4+|Z_1-Z_2|^2}}+\frac{(Z_1-Z_2)(|Z_1|^2-|Z_2|^2)}{2(2-\sqrt{4+|Z_1-Z_2|^2})^2}.
\end{equation}
\end{prop}
\begin{proof}
The result follows from pulling back the metric (\ref{e:metric}) by the transformation (\ref{e:confcoinv1}) and (\ref{e:confcoinv2}) and the result is
\[
ds^2=\frac{1}{1+{\textstyle{\frac{1}{4}}}|Z_1-Z_2|^2}\left(dZ_1d\bar{Z}_1-dZ_2d\bar{Z}_2\right).
\]
\end{proof}
\vspace{0.1in}
In conformal coordinates, reflection is exactly reflection through the origin:
\vspace{0.1in}
\begin{coro}
In the special case where $\gamma_0$ is the $x_3-$axis, the map in conformal coordinates is reflection in the origin:
\[
{\mathcal R}_0(X_1,X_2,X_3,X_4)=(-X_1,-X_2,-X_3,-X_4).
\]
\end{coro}
\begin{proof}
This follows easily from applying Lemma \ref{l:1} to the transformations (\ref{e:confco1}) and (\ref{e:confco2}).
\end{proof}

\vspace{0.1in}

\section{Further Examples of Paracomplex Structures}\label{s:5}
This section provides examples of 4-manifolds with natural almost paracomplex structures, some of which are isometric, some anti-isometric and some parallel. 

\vspace{0.1in}

\subsection{The space of oriented geodesics in real space forms}
Here we consider almost (para)complex structures that arise on spaces of oriented geodesics of 3-manifolds of constant curvature \cite{agk} \cite{GG1} \cite{kahlermetric} \cite{Salvai}.

Let ${\mathbb R}_p^4$ be (pseudo-)Euclidean 4-space $({\mathbb R}^4,\<.,.\>_p)$, where
\[
\<.,.\>_p=-\sum_{i=1}^p dX_i^2+\sum_{i=p+1}^4dX_i^2,
\]
with $p\in\{0,1,2,3\}$.
Let ${\mathbb S}_p^{3}$ be the quadric in ${\mathbb R}^4$ defined by
\[
{\mathbb S}_p^{3}=\{x\in {\mathbb R}^4|\; \<x,x\>_p=1\}.
\]
In particular, the 3-sphere ${\mathbb S}^{3}$ is the quadric ${\mathbb S}_0^{3}$, the hyperbolic 3-space ${\mathbb H}^{3}$ is anti-isometric to ${\mathbb S}_3^{3}\cap \{x\in {\mathbb R}^4|\, X_4>0\}$, the de Sitter 3-space $d {\mathbb S}^{3}$ is the quadric ${\mathbb S}_1^{3}$ and, the anti de Sitter 3-space $Ad {\mathbb S}^{3}$ is anti-isometric to the quadric ${\mathbb S}_2^{3}$.

Using the inclusion map $i:{\mathbb S}_p^{3}\rightarrow {\mathbb R}_p^4$, the induced metric $g_p:=i^{\ast}\<.,.\>_p$ is of constant sectional curvature equal to $1$. The space of oriented geodesics in ${\mathbb S}_p^{3}$ is a 4-dimensional manifold and is identified with the following Grasmmannian spaces of oriented planes on ${\mathbb R}_p^4$:
\[
{\mathbb L}^{\pm}({\mathbb S}_p^{3})=\{x\wedge y\in \Lambda^2({\mathbb R}_p^4)|\; y\in T_x{\mathbb S}_p^{3},\; g_p(y,y)=\pm 1\},
\]

Let $\iota:{\mathbb L}^{\pm}({\mathbb S}_p^{3})\rightarrow \Lambda^2({\mathbb R}_p^4)$, be the inclusion map and let $\left<\left<,\right>\right>_p$ be the flat metric in the 6-manifold $\Lambda^2({\mathbb R}_p^4)$, defined by
\[
\left<\left<u_1\wedge v_1,u_2\wedge v_2\right>\right>_p:=\left<u_1,u_2\right>_p\left<v_1,v_2\right>_p-\left<u_1,v_2\right>_p\left<u_2,v_1\right>_p.
\]
The signature of $\left<\left<,\right>\right>_p$ is $(6,0)$, $(3,3)$, $(4,2)$, $(3,3)$ for $p=0,1,2,3$, respectively. The metric $G_p=\iota^{\ast}\left<\left<,\right>\right>_p$ on  ${\mathbb L}^{\pm}({\mathbb S}_p^{3})$ is Einstein \cite{An}. 

Given an oriented plane $x\wedge y\in {\mathbb L}^{\pm}({\mathbb S}_p^{3})$, the tangent space is identified to
\[
T_{x\wedge y}{\mathbb L}^{\pm}({\mathbb S}_p^{3})=\{x\wedge X+y\wedge Y\in \Lambda^2({\mathbb R}_p^4):\,\, X,Y\in (x\wedge y)^{\bot}\}.
\]
It is not hard to confirm that the following holds on ${\mathbb L}^{\pm}({\mathbb S}_p^{3})$:
\[
G_p(x\wedge X_1+y\wedge Y_1,x\wedge X_2+y\wedge Y_2)=g_p(X_1,X_2)\pm g_p(Y_1,Y_2).
\]
Let $x\wedge y\in {\mathbb L}^{\pm}({\mathbb S}_p^{3})$ and $X\in (x\wedge y)^{\bot}-\{0\}$. Define $J:(x\wedge y)^{\bot}\rightarrow (x\wedge y)^{\bot}$, so that 
$(x,y,X,JX)$ forms an oriented orthogonal frame. 
The following almost (para)complex structures
\[{\mathbb J}(x\wedge X+y\wedge Y)=x\wedge JX+y\wedge JY,\]
\[
{\mathbb J}'(x\wedge X+y\wedge Y)=y\wedge X\mp x\wedge Y,
\]
together with the metric $G_p$, define a (para-)K\"ahler structure on ${\mathbb L}^{\pm}({\mathbb S}_p^{3})$ - see \cite{An}. In particular, in ${\mathbb L}^{\pm}({\mathbb S}_p^{3})$,
\[
{\mathbb J}^2=\begin{cases}\mp(-1)^p\,\mbox{Id},\quad p\in\{1,2\}\\ -\mbox{Id},\qquad\;\;\quad p\in\{0,3\}\end{cases}\qquad ({\mathbb J}')^2=\mp Id.
\]

Define the following almost (para) complex structure on ${\mathbb L}^{\pm}({\mathbb S}_p^{3})$ by
\[
{\mathbb J}^{\ast}=-{\mathbb J}'{\mathbb J}.
\]

The following proposition gives the relationship between the (para)complex structure ${\mathbb J}^{\ast}$ and the  the Hodge star operator $\ast$ defined on the space of bivectors $\Lambda^2({\mathbb R}_p^4)$ in ${\mathbb R}_p^4$.

\vspace{0.1in}
\begin{prop}
The almost (para)complex structure ${\mathbb J}^{\ast}$ is (up to a sign) the Hodge star operator $\ast$ defined on the space of bivectors $\Lambda^2({\mathbb R}_p^4)$ in ${\mathbb R}_p^4$, restricted to the space of oriented geodesics ${\mathbb L}^{\pm}({\mathbb S}^3_p)$.
\end{prop}
\begin{proof}
Identify each vector $\partial/\partial X_k$ with its dual $dX_k$ and therefore each bivector $\partial/\partial X_k\wedge \partial/\partial X_l$ with the 2-form $dX_k\wedge dX_l$ and consider the Hodge star operator $\ast$ as a mapping in the space of bivectors in ${\mathbb R}_p^4$. In particular, if $u\wedge v\in \Lambda^2({\mathbb R}_p^4)$, then $\ast(u\wedge v)$ is defined as:
\[
\ast(u\wedge v)\wedge u\wedge v=\left<\left<u\wedge v,u\wedge v\right>\right>_pe_1\wedge e_2\wedge e_3\wedge e_4,
\]
where $(e_1,e_2,e_3,e_4)$ is an oriented orthonormal frame of ${\mathbb R}_p^4$. 

If $x\wedge y\in {\mathbb L}^{\pm}({\mathbb S}^3_p)$ and $X,Y\in (x\wedge y)^\bot-\{0\}$, then
\[
\ast(x\wedge X)\wedge x\wedge X=-y\wedge JX\wedge x\wedge X,
\]
which implies 
\begin{equation}\label{e:star1}
\ast(x\wedge X)=-y\wedge JX.
\end{equation}
On the other hand,
\[
\ast(y\wedge Y)\wedge y\wedge Y=\pm x\wedge JY\wedge y\wedge Y,
\]
and thus, 
\begin{equation}\label{e:star2}
\ast(y\wedge Y)=\pm x\wedge JY.
\end{equation}
Using (\ref{e:star1}) and (\ref{e:star2}) we finally get
\[
\ast (x\wedge X+y\wedge Y)=-y\wedge JX\pm x\wedge JY,
\]
which shows that $\ast$ is an endomorphism on $T_{x\wedge y}{\mathbb L}^{\pm}({\mathbb S}^3_p)$. It is not hard now to see that the following holds
\[
\left. \ast \right|_{T_{x\wedge y}{\mathbb L}^{\pm}({\mathbb S}^3_p)}=-{\mathbb J}^{\ast}_{x\wedge y}.
\]
\end{proof}

\vspace{0.1in}

\begin{prop}
The triple $({\mathbb J},{\mathbb J}',{\mathbb J}^{\ast})$ on ${\mathbb L}^{\pm}({\mathbb S}_p^3)$ are commuting (para)complex structures and are isometric or anti-isometric in accordance with Table 1 below.
\begin{table}[ht]
\caption{ (Para)complex Structures on Geodesic Spaces}
 \hspace{1em}  \begin{tabular}{| l  |c |c |c |  }
     \hline
      &   $({\mathbb J},G_p)$ & $({\mathbb J}',G_p)$ & $({\mathbb J}^{\ast},G_p)$ \\ \hline 
   ${\mathbb L}^+({\mathbb S}_0^3)$   & complex/isometric & complex/isometric & para/isometric\\ \hline 
   ${\mathbb L}^+({\mathbb S}_1^3)$   & para/anti & complex/isometric & complex/anti\\ \hline 
   ${\mathbb L}^-({\mathbb S}_1^3)$   & complex/isometric & para/anti & complex/anti\\ \hline 
   ${\mathbb L}^+({\mathbb S}_2^3)$   & complex/isometric & complex/isometric & para/isometric\\ \hline 
   ${\mathbb L}^-({\mathbb S}_2^3)$   & para/anti & para/anti & para/isometric\\ \hline 
   ${\mathbb L}^-({\mathbb S}_3^3)$   & complex/isometric & para/anti & complex/anti\\ \hline 
\end{tabular}
\label{table:nonlin} 
\end{table}
All are parallel with respect to the Levi-Civita connection of $G_p$.
\end{prop}
\begin{proof}
The (para)complex structures ${\mathbb J}$ and ${\mathbb J}'$ were introduced in \cite{An} and their properties noted. For ${\mathbb J}^\ast$, a direct computation yields
\begin{enumerate}
\item ${\mathbb J}^{\ast}$ is $G_p$-parallel since ${\mathbb J}$ and ${\mathbb J}'$ are both $G_p$-parallel (see \cite{An}).
\item $({\mathbb J}^{\ast})^2=(-1)^p\, Id$.
\item $G_p({\mathbb J}^{\ast}\cdot,{\mathbb J}^{\ast}\cdot)=(-1)^pG_p(\cdot,\cdot)$,
\end{enumerate}
and the proof follows.
\end{proof}

\vspace{0.1in}

Define the metric $G'_p:=G_p({\mathbb J}^{\ast} .,.)$, so that on ${\mathbb L}^{\pm}({\mathbb S}_p^{3})$
\[
G'_p(x\wedge X_1+y\wedge Y_1,x\wedge X_2+y\wedge Y_2)=\pm (g_p(X_1,JY_2)+g_p(X_2,JY_1)).
\]
The following is known:
\begin{prop}\cite{An}
The metric $G'_p$ is of neutral signature, scalar flat and locally conformally flat.
\end{prop}

\vspace{0.1in}

\subsection{The product of surfaces}

Paracomplex structures are sometimes referred to as almost product structures, as they give a splitting of the tangent bundle of the manifold at each point. Moreover, when the structure is parallel, Main Theorem \ref{t:3} says that the splitting is locally tangent to a pair of foliations by surfaces. In this section the optimal case, where the 4-manifold is a global product of Riemannian surfaces, is considered and the associated paracomplex structures described.

Let $(\Sigma,g)$ be an oriented 2-manifold. If $g$ is Riemannian (Lorentz) there exists a canonical (para)complex structure $j$ defined on $\Sigma$ (respectively) - for more details see \cite{An2} for the Riemannian case and \cite{Ge2} for the Lorentz case. Consider two (Lorentz) Riemannian 2-manifolds $(\Sigma_1,g_1)$ and $(\Sigma_2,g_2)$ together with their canonical (para)complex structures $j_1$ and $j_2$ and the cartesian product $\Sigma_1\times\Sigma_2$, respectively. The projections $\pi_1$ and $\pi_2$ are defined by
\[
\pi_k:\Sigma_1\times\Sigma_2\rightarrow\Sigma_k:(x_1,x_2)\mapsto x_k,
\]
for $k\in\{1,2\}$. The tangent bundle $T(\Sigma_1\times\Sigma_2)$ is identified with $T\Sigma_1\oplus T\Sigma_2$ by 
\[
X\in T(\Sigma_1\times\Sigma_2)\simeq T\Sigma_1\oplus T\Sigma_2 \ni (X_1,X_2).
\]
Define the almost (para)complex structures $J_1$ and $J_2$ on $T(\Sigma_1\times\Sigma_2)$ by
\[
J_2=j_1\oplus j_2,\qquad J_1=j_1\oplus (-j_2),
\]
and the metrics $G_{\pm}$ by
\[
G_+=\pi_1^{\ast}g_1+\pi_2^{\ast}g_2,\qquad G_-=\pi_1^{\ast}g_1-\pi_2^{\ast}g_2.
\]
We adopt the following notation:
\begin{equation}\label{e:gepsilon}
G_{\epsilon}=\begin{cases}G_+,\quad \mbox{if}\;\epsilon=1\\ G_-,\quad \mbox{if}\;\epsilon=-1\end{cases}
\end{equation}
\begin{prop}\cite{Ge1}
The 4-manifolds $(\Sigma_1\times\Sigma_2,G_{\epsilon},J_1)$ and $(\Sigma_1\times\Sigma_2,G_{\epsilon},J_2)$ are both (para)K\"ahler. 
\end{prop}

We now have the following proposition:
\begin{prop}\label{p:propforj}
The 4-manifold $(\Sigma_1\times\Sigma_2,G_{\epsilon})$ admits a parallel isometric paracomplex structure. 
\end{prop}
\begin{proof}
We define the almost paracomplex structure $J=J_1J_2$. Then 
\[
J_1J_2(X)=(-(\pi_1)_{\ast}X),(\pi_2)_{\ast}X)=J_2J_1(X).
\]
Then, $J$ is $G_{\epsilon}$-isometric, since
\[
G_{\epsilon}(JX,JY)=G_{\epsilon}(J_1J_2X,J_1J_2Y)=G_{\epsilon}(J_2X,J_2Y)=G_{\epsilon}(X,Y).
\]
Additionally, $J$ is $G_{\epsilon}$-parallel. Indeed, if $\nabla^{\epsilon}$ is the Levi-Civita connection of $G^{\epsilon}$ then 
\[
\nabla^{\epsilon}_XJY=\nabla^{\epsilon}_XJ_1J_2Y=J_1\nabla^{\epsilon}_XJ_2Y=J_1J_2\nabla^{\epsilon}_XY=J\nabla^{\epsilon}_XY,
\]
which completes the proof.
\end{proof}
\vspace{0.1in}

Let $(\Sigma_1,g_1)$ and $(\Sigma_2,g_2)$ be Riemannian 2-manifolds and consider the metrics $G_\pm=g_1\pm g_2$ as before.

\vspace{0.1in}
\begin{prop}
The neutral metric $G_-$ is conformally flat (and scalar flat) iff the Gauss curvatures of $g_1$ and $g_2$ are equal. 

\end{prop}
\begin{proof}
Let $\kappa_1$ and $\kappa_2$ be the Gauss curvatures of $g_1$ and $g_2$, respectively. A straight-forward calculation shows that
\[
R(G_\pm)=2(\kappa_1\pm\kappa_2) \qquad |W(G_\pm)|^2={\textstyle{\frac{2}{3}
}}(\kappa_1\pm\kappa_2)^2 
\]
\[
|Ric(G_\pm)|^2=2(\kappa_1^2+\kappa_2^2) \qquad  |Ein (G_\pm)|^2=(\kappa_1\mp\kappa_2)^2. 
\]
The first statement follows directly from the first two equalities.
\end{proof}

\vspace{0.1in}

\section{Anti-isometric Almost Paracomplex Structures}\label{s:2}
An almost paracomplex structure $j$ on a 4-manifold $M$ with metric $g$  (of any signature) is said to be {\it anti-isometric} if $g(j\cdot,j\cdot)=-g(\cdot,\cdot)$. In this section anti-isometric almost paracomplex structures are geometrically characterized. The anti-isometric case of Main Theorem \ref{t:2} is also proven.

\vspace{0.1in}
\begin{prop}\label{p:aiso}
An almost paracomplex structure is anti-isometric iff its eigenspaces are totally null. Thus $g$ must have neutral signature.
\end{prop}
\begin{proof}
We work on the tangent space to a fixed point $p\in M$. By assumption, $j$ is an endomorphism of $T_pM$ which has two eigenvalues, $\pm 1$, each with a 2 dimensional eigenspace. Fix an eigenbasis $(V^+_1,V^+_2,V^-_1,V^-_2)$ s.t.
\[
j(V^\pm_a)=\pm V^\pm_a \qquad\qquad{\mbox{for }} a=1,2.
\]
Clearly, an almost paracomplex structure at a point is exactly the same as the choice of a pair of transverse 2-planes (the $\pm 1$-eigenspaces) and hence the space of almost paracomplex structures at a point is 8-dimensional and can be identified with an open set in the product of real Grassmannians $Gr(2,4)\times Gr(2,4)$.

If the almost paracomplex structure is anti-isometric, projecting onto the eigenbasis we find
\[
g(V^+_1,V^+_1)=0 \qquad g(V^+_1,V^+_2)=0 \qquad g(V^+_2,V^+_2)=0 
\]
\[
g(V^-_1,V^-_1)=0 \qquad g(V^-_1,V^-_2)=0 \qquad g(V^-_2,V^-_2)=0.
\]
These are equivalent to the condition that the span of $(V^+_1,V^+_2)$ and the span of $(V^-_1,V^-_2)$ are both totally null planes, as claimed. 
\end{proof}
\vspace{0.1in}

There is a disjoint union of two $S^1$'s worth of totally null planes at a point in a neutral 4-manifold, and hence the space of anti-isometric almost paracomplex structures at a point is two dimensional.

Let $(M,g)$ be a conformally flat neutral 4-manifold with complex conformal coordinates $(Z_1,\bar{Z}_1,Z_2,\bar{Z}_2)$ so that
\[
g=\Omega^2(dZ_1d\bar{Z}_1-dZ_2d\bar{Z}_2),
\]
for some smooth non-zero function $\Omega:M\rightarrow{\mathbb R}$. A straight-forward calculation shows that 

\vspace{0.1in}
\begin{prop}\cite{CG}\label{p:cfsfuhe}
The scalar curvature of a conformally flat neutral metric $g$ is zero iff the conformal factor satisfies the ultra-hyperbolic equation:
\[
(\partial_1\bar{\partial}_1-\partial_2\bar{\partial}_2)\Omega=0.
\]
\end{prop}
\vspace{0.1in}

We now consider an almost paracomplex structure on a conformally flat neutral 4-manifold $(M,g)$. An almost paracomplex structure can be uniquely defined by a transverse splitting of the tangent bundle 
\[
T_p M=P \oplus \hat{P},
\]
into two 2-planes, $P$ and $\hat{P}$ which we declare to be the $+1$ and $-1$ eigenspace of $j$, respectively. We now carry out the construction of $j$ from its eigenspace in two cases: the anti-isometric and the isometric.

Assume now that $j$ is anti-isometric. As we saw, this means that $P$ and $\hat{P}$ are both totally null planes. 
\vspace{0.1in}
\begin{exem}
For fixed $\phi\in[0,2\pi)$ consider the plane $P$ in ${\mathbb R}^{2,2}$ through the origin parameterized by $w\in{\mathbb C}$ 
\[
Z_1=w \qquad\qquad Z_2=we^{i\phi}.
\]
These are totally null planes, called $\alpha-$planes, since the quadratic form $Q=Z_1\bar{Z_1}-Z_2\bar{Z}_2$ evaluates to zero on $P$. Similarly, for fixed $\phi\in[0,2\pi)$ the planes 
\[
Z_1=w \qquad\qquad Z_2=\bar{w}e^{i\phi},
\]
are also totally null planes, called $\beta-$planes.

There is a disjoint union of two circles worth of totally null planes, and every $\alpha-$plane intersects every $\beta-$plane in a null line.
\end{exem}
\vspace{0.1in}
The infinitesimal version of this example at a point $p\in M$ says that the $\alpha-$planes are generated by
\[
{\mbox{Span }}_{\mathbb C}\left( \frac{\partial}{\partial Z_1}+e^{i\phi}\frac{\partial}{\partial Z_2}\right),
\]
and the $\beta-$planes are generated by
\[
{\mbox{Span }}_{\mathbb C}\left( \frac{\partial}{\partial Z_1}+e^{i\phi}\frac{\partial}{\partial \bar{Z}_2}\right).
\]
We construct an anti-isometric paracomplex structure by choosing two transverse totally null planes. As an $\alpha$-plane is never transverse to a $\beta$-plane, we can only have either two $\alpha$-planes or two $\beta$-planes. 

Let us deal with the two cases separately.

\vspace{0.1in}
\subsection{The $\alpha$-almost paracomplex structures}

For two functions $\phi_1,\phi_2:{\mathbb C}^2\rightarrow [0,2\pi)$ we seek the almost paracomplex structure $j$ satisfying
\[
j\left( \frac{\partial}{\partial Z_1}+e^{i\phi_1}\frac{\partial}{\partial {Z}_2}\right)= \frac{\partial}{\partial Z_1}+e^{i\phi_1}\frac{\partial}{\partial {Z}_2}
\]
\[
j\left( \frac{\partial}{\partial Z_1}+e^{i\phi_2}\frac{\partial}{\partial {Z}_2}\right)= -\frac{\partial}{\partial Z_1}-e^{i\phi_2}\frac{\partial}{\partial {Z}_2},
\]
which can be rearranged to
\begin{equation}\label{e:alapcs1}
j\left( \frac{\partial}{\partial Z_1}\right)= \left(1-\frac{2e^{i\phi_1}}{e^{i\phi_1}-e^{i\phi_2}}\right)\frac{\partial}{\partial Z_1}+\left(1-\frac{e^{i\phi_1}+e^{i\phi_2}}{e^{i\phi_1}-e^{i\phi_2}}\right)e^{i\phi_1}\frac{\partial}{\partial {Z}_2}
\end{equation}
\begin{equation}\label{e:alapcs2}
j\left(\frac{\partial}{\partial {Z}_2}\right)= \frac{2}{e^{i\phi_1}-e^{i\phi_2}}\frac{\partial}{\partial Z_1}+\frac{e^{i\phi_1}+e^{i\phi_2}}{e^{i\phi_1}-e^{i\phi_2}}\frac{\partial}{\partial {Z}_2}.
\end{equation}

\vspace{0.1in}
\begin{prop}
An $\alpha$ almost paracomplex structures is parallel with respect to a conformally flat metric iff
\[
\partial_1\left(\Omega e^{-i\phi_1}\right)=-\partial_2\Omega \qquad\qquad \bar{\partial}_2\left(\Omega e^{-i\phi_1}\right)=-\bar{\partial}_1\Omega
\]
\[
\partial_1\left(\Omega e^{-i\phi_2}\right)=-\partial_2\Omega \qquad\qquad \bar{\partial}_2\left(\Omega e^{-i\phi_2}\right)=-\bar{\partial}_1\Omega.
\]
As a consequence
\[
(\partial_1\bar{\partial}_1-\partial_2\bar{\partial}_2)\Omega=0,
\]
and the neutral metric must be scalar flat.
\end{prop}
\begin{proof}
Equations (\ref{e:alapcs1}) and (\ref{e:alapcs2}) define the components of the almost paracomplex structure $j=j_m^n$ in conformal coordinates in terms of the functions $\phi_1,\phi_2$, and $j$ being parallel with respect to the Levi-Civita connection of $g$ means
\[
\nabla_lj_m^n=\partial_lj_m^n-j_k^n\Gamma_{ml}^k+j_m^k\Gamma_{kl}^n=0.
\]
Here the gradient of the conformal factor $\Omega$ enters the equations through the Christoffel symbols $\Gamma_{mn}^l$. The resulting equations can be rearranged to the stated set.
\end{proof}
\vspace{0.1in}

\subsection{The $\beta$-almost paracomplex structures}

Similarly, for two functions $\phi_1,\phi_2:{\mathbb C}^2\rightarrow [0,2\pi)$ we seek the $\alpha$ almost paracomplex structure $j$ satisfying
\[
j\left( \frac{\partial}{\partial Z_1}+e^{i\phi_1}\frac{\partial}{\partial \bar{Z}_2}\right)= \frac{\partial}{\partial Z_1}+e^{i\phi_1}\frac{\partial}{\partial \bar{Z}_2}
\]
\[
j\left( \frac{\partial}{\partial Z_1}+e^{i\phi_2}\frac{\partial}{\partial \bar{Z}_2}\right)= -\frac{\partial}{\partial Z_1}-e^{i\phi_2}\frac{\partial}{\partial \bar{Z}_2}
\]
which can be rearranged to
\[
j\left( \frac{\partial}{\partial Z_1}\right)= \left(1-\frac{2e^{i\phi_1}}{e^{i\phi_1}-e^{i\phi_2}}\right)\frac{\partial}{\partial Z_1}+\left(1-\frac{e^{i\phi_1}+e^{i\phi_2}}{e^{i\phi_1}-e^{i\phi_2}}\right)e^{i\phi_1}\frac{\partial}{\partial \bar{Z}_2}
\]
\[
j\left(\frac{\partial}{\partial \bar{Z}_2}\right)= \frac{2}{e^{i\phi_1}-e^{i\phi_2}}\frac{\partial}{\partial Z_1}+\frac{e^{i\phi_1}+e^{i\phi_2}}{e^{i\phi_1}-e^{i\phi_2}}\frac{\partial}{\partial \bar{Z}_2}.
\]
\vspace{0.1in}
\begin{prop}
A $\beta$-almost paracomplex structures is parallel with respect to a conformally flat metric iff
\[
\partial_1\left(\Omega e^{-i\phi_1}\right)=-\bar{\partial}_2\Omega \qquad\qquad {\partial}_2\left(\Omega e^{-i\phi_1}\right)=-\bar{\partial}_1\Omega
\]
\[
\partial_1\left(\Omega e^{-i\phi_2}\right)=-\bar{\partial}_2\Omega \qquad\qquad {\partial}_2\left(\Omega e^{-i\phi_2}\right)=-\bar{\partial}_1\Omega.
\]
As a consequence
\[
(\partial_1\bar{\partial}_1-\partial_2\bar{\partial}_2)\Omega=0,
\]
and the neutral metric must be scalar flat.
\end{prop}
\begin{proof}
The proof is identical to the $\alpha$-almost paracomplex case.
\end{proof}\vspace{0.1in}

This completes the proof of Main Theorem \ref{t:2} for anti-isometric paracomplex structures.
\vspace{0.1in}

\section{Isometric Almost Paracomplex Structures}\label{s:3}
An almost paracomplex structure $j$ on a 4-manifold $M$ with metric $g$  (of any signature) is said to be {\it isometric} if $g(j\cdot,j\cdot)=g(\cdot,\cdot)$. In this section we geometrically characterize isometric almost paracomplex structures and prove the isometric case of Main Theorem \ref{t:2}.

\vspace{0.1in}
\begin{prop}\label{p:iso}
An almost paracomplex structure is isometric iff its eigenspaces are orthogonal.
\end{prop}
\begin{proof}
As in the proof of Proposition \ref{p:aiso} fix a point $p\in M$ and an eigenbasis $(V^+_1,V^+_2,V^-_1,V^-_2)$ for $T_pM$ s.t.
\[
j(V^\pm_a)=\pm V^\pm_a \qquad\qquad{\mbox{for }} a=1,2.
\]
The space of almost paracomplex structures at a point is 8-dimensional and can be identified with an open set in the product of real Grassmannians $Gr(2,4)\times Gr(2,4)$.

If the almost paracomplex structure is isometric, projecting onto the eigenbasis we find
\[
g(V^+_1,V^-_1)=0 \qquad g(V^+_1,V^-_2)=0 \qquad g(V^+_2,V^-_1)=0 \qquad g(V^+_2,V^-_2)=0.
\]
These are equivalent to the condition that the span of $(V^+_1,V^+_2)$ is orthogonal to the span of $(V^-_1,V^-_2)$, as claimed. 

\end{proof}
\vspace{0.1in}

Thus an isometric almost paracomplex structure at a point is uniquely determined by a single 2-plane or a point in the 4-dimensional Grassmannian $Gr(2,4)$.

\begin{exem}
Consider a plane $P$ in ${\mathbb R}^{2,2}$ which is a graph over $Z_1$ so that, for some $\alpha,\beta\in{\mathbb C}$, it can be parameterized as
\[
Z_1=w \qquad\qquad Z_2=\alpha w+\beta\bar{w}.
\]
A simple calculation shows that the plane orthogonal to $P$ w.r.t. the neutral metric can be parameterized
\[
Z_1=w \qquad\qquad Z_2=\frac{\alpha}{\alpha\bar{\alpha}-\beta\bar{\beta}} w-\frac{\beta}{\alpha\bar{\alpha}-\beta\bar{\beta}}\bar{w}.
\]
\end{exem}
The infinitesimal version of this example at a point $p\in M$ says that the plane $P$ is generated by
\[
{\mbox{Span }}_{\mathbb C}\left( \frac{\partial}{\partial Z_1}+\alpha\frac{\partial}{\partial Z_2}+\bar{\beta}\frac{\partial}{\partial \bar{Z}_2}\right),
\]
for some $\alpha,\beta\in{\mathbb C}$, and the orthogonal plane is generated by
\[
{\mbox{Span }}_{\mathbb C}\left( \frac{\partial}{\partial Z_1}+\frac{\alpha}{\Delta_1}\frac{\partial}{\partial Z_2}-\frac{\bar{\beta}}{\Delta_1}\frac{\partial}{\partial \bar{Z}_2}\right),
\]
where for ease of notation we have introduced $\Delta_1=\alpha\bar{\alpha}-\beta\bar{\beta}$.

The almost paracomplex structure is now defined by the conditions
\[
j\left( \frac{\partial}{\partial Z_1}+\alpha\frac{\partial}{\partial Z_2}+\bar{\beta}\frac{\partial}{\partial \bar{Z}_2}\right)= \frac{\partial}{\partial Z_1}+\alpha\frac{\partial}{\partial Z_2}+\bar{\beta}\frac{\partial}{\partial \bar{Z}_2}
\]
\[
j\left( \frac{\partial}{\partial Z_1}+\frac{\alpha}{\Delta_1}\frac{\partial}{\partial Z_2}-\frac{\bar{\beta}}{\Delta_1}\frac{\partial}{\partial \bar{Z}_2}\right)=-\frac{\partial}{\partial Z_1}-\frac{\alpha}{\Delta_1}\frac{\partial}{\partial Z_2}+\frac{\bar{\beta}}{\Delta_1}\frac{\partial}{\partial \bar{Z}_2}.
\]
Introducing the notation
\[
\Delta_2=\alpha\bar{\alpha}(1-\Delta_1^{-1})^2-\beta\bar{\beta}(1+\Delta_1^{-1})^2,
\]
these imply
\begin{align*}
j\left( \frac{\partial}{\partial Z_1}\right)=&\left(1-\frac{2[\alpha\bar{\alpha}(\Delta_1-1)-\beta\bar{\beta}(\Delta_1+1)]}{\Delta_1\Delta_2}\right)\frac{\partial}{\partial Z_1}+\frac{4\alpha\bar{\beta}}{\Delta_1\Delta_2}\frac{\partial}{\partial \bar{Z}_1}\\
&+\frac{\alpha\left(1+(1+4\beta\bar{\beta}-\Delta_1^2)\right)}{\Delta_1\Delta_2}\frac{\partial}{\partial Z_2}+\frac{\alpha\left(1+(1+4\alpha\bar{\alpha}-\Delta_1^2)\right)}{\Delta_1\Delta_2}\frac{\partial}{\partial \bar{Z}_2}
\end{align*}
\begin{equation*}
j\left( \frac{\partial}{\partial Z_2}\right)=\frac{2\bar{\alpha}(\Delta_1-1)}{\Delta_1\Delta_2}\frac{\partial}{\partial Z_1}-\frac{2\beta(\Delta_1+1)}{\Delta_1\Delta_2}\frac{\partial}{\partial \bar{Z}_1}
+\frac{(\Delta_1^2-1)}{\Delta_1\Delta_2}\frac{\partial}{\partial Z_2}+\frac{4\bar{\alpha}\bar{\beta}}{\Delta_1\Delta_2}\frac{\partial}{\partial \bar{Z}_2}.
\end{equation*}

\vspace{0.1in}
\begin{prop}
An isometric almost paracomplex structure is parallel with respect to a neutral conformally flat metric iff the complex functions $\alpha$ and $\beta$ satisfy the system:
\begin{equation}\label{e:eq1}
\Omega \partial_1\alpha-\alpha\partial_1\Omega =\alpha(\alpha\partial_2\Omega+\bar{\beta}\bar{\partial}_2\Omega)
\end{equation}
\begin{equation}\label{e:eq2}
\Omega \partial_1\bar{\alpha}+\bar{\alpha}\partial_1\Omega =(\beta\bar{\beta}-1)\partial_2\Omega+\bar{\alpha}\bar{\beta}\bar{\partial}_2\Omega
\end{equation}
\begin{equation}\label{e:eq3}
\Omega \partial_2\alpha+\alpha\partial_2\Omega =(\beta\bar{\beta}-1)\partial_1\Omega+{\alpha}\bar{\beta}\bar{\partial}_1\Omega
\end{equation}
\begin{equation}\label{e:eq4}
\Omega \partial_2\bar{\alpha}-\bar{\alpha}\partial_2\Omega =\bar{\alpha}(\bar{\alpha}\partial_1\Omega+\bar{\beta}\bar{\partial}_1\Omega),
\end{equation}
and
\begin{equation}\label{e:eq5}
\Omega \partial_1\beta+\beta\partial_1\Omega =\alpha\beta\partial_2\Omega+(\alpha\bar{\alpha}-1)\bar{\partial}_2\Omega
\end{equation}
\begin{equation}\label{e:eq6}
\Omega \partial_1\bar{\beta}-\bar{\beta}\partial_1\Omega =\bar{\beta}(\alpha\partial_2\Omega+\bar{\beta}\bar{\partial}_2\Omega)
\end{equation}
\begin{equation}\label{e:eq7}
\Omega \partial_2\beta+\beta\partial_2\Omega =\bar{\alpha}\beta\partial_1\Omega+(\alpha\bar{\alpha}-1)\bar{\partial}_1\Omega
\end{equation}
\begin{equation}\label{e:eq8}
\Omega \partial_2\bar{\beta}-\bar{\beta}\partial_2\Omega =\bar{\beta}(\bar{\alpha}\partial_1\Omega+\bar{\beta}\bar{\partial}_1\Omega).
\end{equation}
\end{prop}
\begin{proof}
The proof is a direct computation identical to the anti-isometric cases and will be omitted.
\end{proof}
\vspace{0.1in}
\noindent{\bf Proof of Main Theorem} \ref{t:2}:

We now show that a neutral conformally flat 4-manifold that admits an isometric parallel almost paracomplex structure is scalar flat.

Introducing functions $a,b,\phi,\theta$ defined $\alpha=ae^{i\theta}$ and $\beta=be^{i\phi}$, the parallel conditions become
\[
2\Omega\partial_1 a=(a^2+b^2-1)e^{i\theta}\partial_2\Omega+2abe^{-i\phi}\bar{\partial}_2\Omega
\]
\[
2\Omega\partial_1 b=2abe^{i\theta}\partial_2\Omega+(a^2+b^2-1)e^{-i\phi}\bar{\partial}_2\Omega
\]
\[
2\Omega\partial_2 a=(a^2+b^2-1)e^{-i\theta}\partial_1\Omega+2abe^{-i\phi}\bar{\partial}_1\Omega
\]
\[
2\Omega\partial_2 b=2abe^{-i\theta}\partial_1\Omega+(a^2+b^2-1)e^{-i\phi}\bar{\partial}_1\Omega
\]
\[
2ai\Omega\partial_1 \theta=2a\partial_1\Omega+(a^2-b^2+1)e^{i\theta}{\partial}_2\Omega
\]
\[
2bi\Omega\partial_1 \phi=-2b\partial_1\Omega+(a^2-b^2-1)e^{-i\phi}\bar{\partial}_2\Omega
\]
\[
2ai\Omega\partial_2 \theta=-(a^2-b^2+1)e^{-i\theta}\partial_1\Omega-2a{\partial}_2\Omega
\]
\[
2bi\Omega\partial_2 \phi=(a^2-b^2+1)e^{-i\phi}\bar{\partial}_1\Omega-2b{\partial}_2\Omega.
\]
The integrability conditions imposed by the last 4 equations, namely that $\partial_1\partial_2\phi-\partial_2\partial_1\phi=0$ and $\partial_1\bar{\partial}_2\theta-\bar{\partial}_2\partial_1\theta=0$ means that
\[
(\partial_1\bar{\partial}_1-\partial_2\bar{\partial}_2)\Omega=0.
\]
For the conformal factor $\Omega$ of a conformally flat metric, satisfying this equation is equivalent to the scalar curvature of the metric being zero - see Proposition \ref{p:cfsfuhe}.
\qed

\vspace{0.1in}

\section{Differential Invariants of 2-plane Fields}\label{s:4}

This section explores the first order differential invariants of 2-plane fields on a pseudo-Riemannian 4-manifold and contains the proof of Main Theorem \ref{t:3}. The treatment of null geodesic congruences in general relativity as introduced by Newman and Penrose \cite{NaP} is adapted to 2-plane fields on 4-manifolds.

Let $P$ be a 2-plane field on a pseudo-Riemannian 4-manifold $(M,g)$, a plane that arises as an eigenplane of an isometric almost paracomplex structure. From Proposition \ref{p:iso} an almost paracomplex structure is isometric iff its eigenspaces are mutually orthogonal (and transverse) and so the metric induced on $P$ must be non-degenerate. This leads to different cases depending on the signature of the induced metric and the signature of the ambient metric.

In what follows, we consider in detail the case where $g$ is either Riemannian or neutral, and the metric induced on $P$ by $g$ is definite. We call such a plane {\it definite} and note that it is a vacuous condition if the metric is Riemannian.

Let $\hat{P}$ be the 2-plane field on $M$ that is orthogonal to $P$, which will again be definite. Hatted quantities will refer to this normal 2-plane field.

Choose an orthonormal frame $\{e_\mu\}=\{e_1,e_2,e_{\hat{1}},e_{\hat{2}}\}$ so that $\{e_a\}=\{e_1,e_2\}$  spans $P$ and $\{e_{\hat{a}}\}=\{e_{\hat{1}},e_{\hat{2}}\}$ span $\hat{P}$. Let  $\{\theta^\mu\}=\{\theta^1,\theta^2,\theta^{\hat{1}},\theta^{\hat{2}}\}$ be the dual orthonormal co-frame of 1-forms.

The Levi-Civita connection $\nabla$ has coefficients associated with this frame 
\[
\Gamma_{\mu\nu\alpha}=g(\nabla_\mu\;e_\nu,e_\alpha),
\]
where $\mu,\nu,\alpha$ vary over $(1,2,\hat{1},\hat{2})$. Note that, due to orthonormality of the frames
\[
\Gamma_{\mu\nu\alpha}=-\Gamma_{\mu\alpha\nu}.
\]
Define the double null complex frame associated with $P$ and $\hat{P}$ by introducing 
\[
e_\pm={\textstyle{\frac{1}{\sqrt{2}}}}\left(e_1\mp ie_2\right) \qquad\qquad e_{\hat{\pm}}={\textstyle{\frac{1}{\sqrt{2}}}}\left(e_{\hat{1}}\mp ie_{\hat{2}}\right), 
\]
Now we allow Greek indices $\mu,\nu,...$ to range over $\{+,-,\hat{+},\hat{-}\}$ and tangent latin indices $a,b,...$ over $\{+,-\}$ and normal latin indices $\hat{a},\hat{b},...$ over $\{\hat{+},\hat{-}\}$.
\vspace{0.1in}
\begin{defi}
For a plane field $P$ the following complex connection coefficients 
\[
\lambda=\Gamma_{+\hat{+}-}-\Gamma_{-\hat{+}+}
\qquad
\rho=\Gamma_{+\hat{+}-}+\Gamma_{-\hat{+}+}
\qquad
\sigma_+=\Gamma_{+\hat{+}+}
\qquad
\sigma_-=\Gamma_{-\hat{+}-},
\]
we refer to as the {\em complex twist}, {\em complex divergence}, {\em positive complex shear} and {\em negative complex shear}, respectively.
\end{defi}
\vspace{0.1in}
The first of these has an obvious interpretation via the Frobenius' Theorem and the others also have geometric significance:
\vspace{0.1in}
\begin{lemm}
The 2-plane field $P$ is integrable iff the complex twist vanishes: $\lambda=0$.

If the complex twist vanishes, there exist surfaces in $M$ tangent to $P$ and the surfaces are minimal iff the integrable plane fields also have vanishing complex divergence $\rho=0$.

The surfaces are totally umbilic iff the integrable plane fields also have vanishing complex shears $\sigma_+=\sigma_-=0$. 

The Gauss curvature of the surfaces is $\kappa=|\rho|^2-|\sigma_+|^2-|\sigma_-|^2$.

\end{lemm}
\begin{proof}
By definition and symmetry
\[
\lambda=\Gamma_{+\hat{+}-}-\Gamma_{-\hat{+}+}=-\Gamma_{+-\hat{+}}+\Gamma_{-+\hat{+}}=g(\nabla_-\;e_+-\nabla_+\;e_-,e_{\hat{+}})
=g([e_-,e_+],e_{\hat{+}}),
\]
and similarly
\[
\bar{\lambda}=g([e_+,e_-],e_{\hat{-}}).
\]
This is zero precisely when the plane is integrable. In this case, the normal-valued second fundamental form is
\[
A(X,Y)=\begin{bmatrix}
X^+ & X^-
\end{bmatrix}
\begin{bmatrix}
\sigma_+\;e_{\hat{+}}+\bar{\sigma}_-\;e_{\hat{-}} & \rho\;e_{\hat{+}}+\bar{\rho}\;e_{\hat{-}}\\
\rho\;e_{\hat{+}}+\bar{\rho}\;e_{\hat{-}} & \sigma_-\;e_{\hat{+}}+\bar{\sigma}_+\;e_{\hat{-}} 
\end{bmatrix}\begin{bmatrix}
Y^+ \\
Y^-
\end{bmatrix}.
\]
The mean curvature vector is the trace
\[
H=\rho e_{\hat{+}}+\bar{\rho} e_{\hat{-}},
\]
and so the surfaces are minimal iff $\rho=0$.

If $\sigma_+=\sigma_-=0$ then 
\[
A(X,Y)=\left(\rho\;e_{\hat{+}}+\bar{\rho}\;e_{\hat{-}}\right)\left(X^+Y^-+X^-Y^+\right)\propto g(X,Y),
\]
as claimed.
\end{proof}
\vspace{0.1in}

Define the same invariants for the orthogonal 2-plane field $\hat{P}$ to yield
\[
\hat{\lambda}=\Gamma_{\hat{+}+\hat{-}}-\Gamma_{\hat{-}+\hat{+}}
\qquad
\hat{\rho}=\Gamma_{\hat{+}+\hat{-}}+\Gamma_{\hat{-}+\hat{+}}
\qquad
\hat{\sigma}_+=\Gamma_{\hat{+}+\hat{+}}
\qquad
\hat{\sigma}_-=\Gamma_{\hat{-}+\hat{-}}.
\]

\vspace{0.1in}
\begin{prop}
An isometric almost paracomplex structure on a Riemannian or neutral 4-manifold with definite eigenplanes, is parallel along an eigenplane iff all first order differential invariants of the eigenplane distribution vanish (the complex twist, complex divergence and both complex shears are identically zero). In particular, the eigenplane distribution is integrable.
\end{prop}
\begin{proof}
Given a definite 2-plane field $P$, consider the isometric almost paracomplex structure $J$ whose $+1$-eigenspace is $P$. In terms of frames, $j$ can be written
\[
j=e_a\otimes\theta^a-e_{\hat{a}}\otimes\theta^{\hat{a}},
\]
with summation over the two-dimensional indices $a,\hat{a}$.

The almost paracomplex $j$ structure is parallel along a plane field $P$ if $\nabla_X j=0$ for all $X\in P$. Thus we restrict our covariant derivative to tangent indices. 

The remaining two indices can be projected onto tangent-tangent, tangent-normal, normal-tangent and normal-normal components. The first and last automatically vanish by the computation
\[
g(\nabla_\mu j\;(e_a),e_b)=\Gamma_{\mu ab}+\Gamma_{\mu ba}=0,
\]
and
\[
g(\nabla_\mu j\;(e_{\hat{a}}),e_{\hat{b}})=\Gamma_{\mu {\hat{a}}{\hat{b}}}+\Gamma_{\mu {\hat{b}}{\hat{a}}}=0.
\]
Thus the only requirement for $j$ to be parallel along $P$ is that
\[
g(\nabla_a j\;(e_{\hat{a}}),e_b)=\Gamma_{ab\hat{a}}-\Gamma_{a\hat{a}b}=2\Gamma_{ab\hat{a}}=0,
\]
and
\[
g(\nabla_a j\;(e_b),e_{\hat{b}})=\Gamma_{a\hat{b}b}-\Gamma_{ab\hat{b}}=2\Gamma_{a\hat{b}b}=0.
\]
These are exactly the vanishing of the complex connection coefficients $\lambda,\rho,\sigma_+,\sigma_-$. 
\end{proof}

Thus the eigenplanes are integrable and the second fundamental form of the integral surfaces vanishes. 

At this point we have established Main Theorem \ref{t:3} in the case of definite planes in Riemannian or neutral 4-manifolds. The remaining cases of both  $P$ and $\hat{P}$ being Lorentz in a neutral 4-manifold, and $P$ Lorentz and $\hat{P}$ definite (or vice versa) in a Lorentz 4-manifold, can be handled similarly by defining analogous paracomplex twist, paracomplex divergence and paracomplex shears via double null paracomplex frames. The proof is then identical.

This completes the proof of Main Theorem \ref{t:3}.  

\vspace{0.1in}

\section{Einstein 4-manifolds and Paracomplex Structures}\label{s:6}
In this section we prove Main Theorem \ref{t:4}: a parallel isometric paracomplex structure provides a duality between Riemannian Einstein metrics and neutral conformally flat, scalar flat metrics on 4-manifolds. 

We start with the following proposition: 

\begin{prop}\label{p:main}
Let $(M,g)$ be a Riemannian 4-manifold and $j$ be an isometric almost paracomplex structure. Then, the metric $g'(\cdot,\cdot)=g(j\cdot,\cdot)$ is of neutral signature. If furthermore $j$ is parallel, then the corresponding Ricci tensors $Ric$ and $Ric'$ are equal.
\end{prop}
\begin{proof}
Since $j$ is an almost paracomplex structure on $M$, the eigenbundles $T_+M$ and $T_-M$ corresponding to the eigenvalues $+1$ and $-1$ respectively, have rank equal to 2. 

Let $\{v_1,v_2, v_3,v_4\}$ be $g$-orthonormal eigenvectors of $j$ such that
\[
jv_k=v_k,\qquad jv_{2+k}=-v_{2+k},
\]
where $k=1,2$. Then $g'(v_k,v_k)=-g'(v_{2+k},v_{2+k})=1$ and $g'(v_k,v_l)=0$ for any $k\neq l$. This means that $g'$ is of neutral signature. 

For $k=1,2$, define the following vector fields:
\[
e_k=\textstyle{\frac{1}{\sqrt{2}}}(v_k+v_{2+k}),\qquad e_{k+2}=je_k=\frac{1}{\sqrt{2}}(v_k-v_{2+k}),
\]
so that $\{e_1, e_2,e_{3}:=je_1,e_{4}:=je_2\}$ form a $g$-orthonormal frame. Thus,
\[
g'_{kl}:=g'(e_k,e_l)=0,\qquad g'_{k\;2+l}:=g'(e_k,e_{2+l})=\delta_{kl},
\]
where $k,l=1,2$.

The fact that $j$ is parallel implies that $g$ and $g'$ share the same Levi-Civita connection $\nabla$ (see Lemma 3.2 in \cite{An}).
So, the corresponding Riemann curvature tensors must be equal, i.e., $R=R'$. Hence,
\[
g'(R'(u,v)z,w)=g(R(u,v)z,jw).
\]
Letting $R(u,v,z,w):=g(R(u,v)z,w)$ and $R'(u,v,z,w):=g'(R'(u,v)z,w)$, we then have 
\[
R'(u,v,z,w)=R(u,v,z,jw).
\]
Let $Ric$ and $Ric'$ be the Ricci tensors of $g$ and $g'$, respectively. Then,
\begin{eqnarray}
Ric'(X,Y)&=&\sum_{k,l=1}^2(g')^{k\;2+l}\left(R'(X,e_k,Y,e_{2+l})+R'(X,e_{2+l},Y,e_{k})\right)\nonumber \\
&=&\sum_{k,l=1}^2\delta_{kl}(R(X,e_k,Y,e_{l})+R(X,e_{2+l},Y,e_{2+k}))\nonumber \\
&=&Ric(X,Y),\nonumber
\end{eqnarray}
which completes the proof.
\end{proof}

\vspace{0.1in}

Main Theorem \ref{t:4} says that a parallel, isometric paracomplex structure maps Einstein metrics to locally conformally flat neutral metrics and the other way round. 

\vspace{0.1in}

\noindent{\bf Proof of Main Theorem} \ref{t:4}:

Suppose that the metric $g$ is Einstein. We know from Theorem \ref{t:2} that the scalar curvature of $g'$ is 0. Then, if $W'$ is the Weyl tensor $g'$, a brief computation shows
\begin{eqnarray}
W'(x,y,z,w)&=&R(x,y,z,jw)+\textstyle{\frac{S}{8}}\left(g(x,w)g(y,jz)+g(y,z)g(x,jw)\right)\nonumber\\
&&\qquad\qquad-\textstyle{\frac{S}{8}}\left(g(x,z)g(y,jw)+g(y,w)g(x,jz)\right).\label{e:weyltensor} 
\end{eqnarray}

To prove that $g'$ is locally conformally flat we need to show that $W'$ vanishes. Using the $g$-orthonormal eigenvectors obtained in Proposition \ref{p:main}, together with (\ref{e:weyltensor}) we have
\[
W'(v_1,v_2,v_1,v_2)=R(v_1,v_2,v_1,jv_2)+\textstyle{\frac{S}{4}}=R(v_1,v_2,v_1,v_2)+\textstyle{\frac{S}{4}}.
\]
Note that $R(v_1,v_3,v_1,v_3)=0$. Indeed,
\[
R(v_1,v_3,v_1,v_3)=-R(v_1,v_3,v_1,jv_3)
=-R'(v_1,v_3,v_1,v_3)
=-R(v_1,v_3,v_1,v_3).\]
Similarly, one can show that the only nonvanishing components of the Riemann curvature tensor are:
\[
R(v_1,v_2,v_1,v_2)=R(v_3,v_4,v_3,v_4)=\textstyle{\frac{S}{4}}.
\]
Indeed,
\[
\textstyle{\frac{S}{4}}= Ric(v_1,v_1)=R(v_1,v_2,v_1,v_2).
\]
It is not hard to verify further that $R(v_3,v_4,v_3,v_4)=\textstyle{\frac{S}{4}}$.
This implies that
\[
W'(v_1,v_2,v_1,v_2)=R(v_1,v_2,v_1,v_2)-\textstyle{\frac{S}{4}}=0,
\]
and
\[
W'(v_3,v_4,v_3,v_4)=-R(v_3,v_4,v_3,v_4)+\textstyle{\frac{S}{4}}=0.
\]
The vanishing of all other components of the Weyl tensor $W'$ can be shown very easily using (\ref{e:weyltensor}). Thus $g'$ is conformally flat.

\vspace{0.1in}

Conversely, suppose that  $g'$ is conformally flat. We now show that $g$ is Einstein. Theorem \ref{t:2} tells us that $g'$ is scalar flat and Proposition \ref{p:main} yields
\[
R(x,y,z,jw)=\textstyle{\frac{1}{2}}(Ric(x,z)g(y,jw)+Ric(y,w)g(x,jz))
\]
\[
\qquad\qquad\qquad\qquad\qquad\qquad-\textstyle{\frac{1}{2}}(Ric(x,w)g(y,jz)-Ric(y,z)g(x,jw)).
\]
Using the above one can show that the only non-vanishing components of the Riemann curvature tensor are:
\[
R(v_1,v_2,v_1,v_2)=R(v_3,v_4,v_3,v_4)=\textstyle{\frac{S}{4}}.
\]
Then, it can be easily seen that
\[
Ric(v_k,v_l)=\textstyle{\frac{S}{4}}\,g(v_k,v_l)\delta_{kl}
\]
and therefore $g$ is Einstein.
\qed

\vspace{0.1in}

Note that Theorem \ref{t:4} can be extended to any signature of the metric $g$. 

Consider the product metric $G_{\epsilon}$ on $\Sigma_1\times\Sigma_2$ defined in (\ref{e:gepsilon}). We now have the following corollary:
\begin{coro}
Consider the product $\Sigma_1\times\Sigma_2$. The following statements are equivalent:
\begin{enumerate}
    \item The Gauss curvatures $k_1,k_2$ of the metrics $g_1,g_2$, respectively, are constants with $k_1=-\epsilon k_2$.
    \item The metric $G_{\epsilon}$ is locally conformally flat and scalar flat.
    \item The metric $G_{-\epsilon}$ is Einstein.
\end{enumerate}
\end{coro}
\begin{proof}
In \cite{Ge1} it is shown that the statements (1) and (2) are equivalent. Note that 
\[
G_{-\epsilon}=G_{\epsilon}(J.,.),
\]
where $J$ is the paracomplex structure defined in Proposition \ref{p:propforj}. This shows the equivalence of (2) and (3) by using Main Theorem \ref{t:4}, and the corollary follows.
\end{proof}

\vspace{0.1in}

\section{Closed 4-manifolds}\label{s:7}
This section explores topological obstructions to parallel paracomplex structures on closed 4-manifolds and contains the proofs of Main Theorems \ref{t:5} and \ref{t:6}.

For a smooth closed, oriented, connected pseudo-Riemannian 4-manifold $(M,g)$ the total integral of certain combinations of curvatures yield topological invariants of the 4-manifold, the simplest of which is the {\em Euler number} $\chi(M)$. This may be defined 
\[
\chi(M)=\sum_{n=0}^4(-1)^n{\rm dim}\;H^n(M)=2-2b_1+b_2,
\]
where $H^n(M)$ are the cohomology groups of $M$, and $b_1,b_2$  the Betti numbers ($b_0=b_4=1$ and $b_3=b_1$ by Poincar\'e duality). The Chern-Gauss-Bonnet Theorem states that
\[
\chi(M)={{\frac{\epsilon}{32\pi^2}}}\int_M|W(g)|^2-2|Ric(g)|^2+{\textstyle{\frac{2}{3}}}S^2\;d^4V_g,
\]
for any metric $g$ of definite signature ($\epsilon=1$) or neutral signature 
($\epsilon=-1$) \cite{Law}. Splitting the Weyl curvature tensor $W$ into its self-dual and anti-self-dual parts one finds that
\[
\chi(M)={{\frac{\epsilon}{32\pi^2}}}\int_M|W^+|^2+|W^-|^2-2|Ric|^2+{\textstyle{\frac{2}{3}}}S^2\;d^4V_g.
\]
The {\em signature} $\tau(M)$ of $M$ is another topological invariant based on the splitting of the intersection form on $H^2(M)$ into positive and negative eigenspaces, so that $b_2=b_++b_-$ and so
\[
\tau(M)=b_+-b_-={{\frac{1}{48\pi^2}}}\int_M|W^+|^2-|W^-|^2\;d^4V_g.
\]

There are well-known topological obstructions to the existence of a Riemannian Einstein metric on a closed 4-manifold $M$. In particular, such a 4-manifold must satisfy the Hitchin-Thorpe inequality \cite{hitch} \cite{MaL} \cite{thorpe}:
\begin{equation}\label{e:htid}
\chi(M)\geq {\scriptstyle{\frac{3}{2}}}|\tau(M)|,
\end{equation}
and in particular, $\chi\geq 0$. 

For closed 4-manifolds the existence of a neutral metric (without any curvature assumption) already implies a restriction on the topology. It is known that the existence of a metric of neutral signature on a closed 4-manifold is equivalent to the existence of a field of oriented tangent 2-planes on the manifold \cite{Mats91}. Moreover

\vspace{0.1in}
\begin{theo} \cite{HaH} \cite{kam02} \cite{Mats91}
Let $M$ be a closed 4-manifold admitting a neutral metric. Then
\begin{equation}\label{e:mats}
\chi(M)+\tau(M)=0{\mbox{ mod }}4
\qquad{\mbox and }\qquad
\chi(M)-\tau(M)=0{\mbox{ mod }}4.
\end{equation}

If $M$ is simply connected, these conditions are sufficient for the existence of a neutral metric.
\end{theo}
\vspace{0.1in}

Thus, neither ${\mathbb S}^4$ nor ${\mathbb C}P^2$ admit a neutral metric, while the K3 manifold does. If the neutral metric is required to be K\"ahler with respect to a compatible complex structure, then there are even more restrictions \cite{pet}. 

Using Main Theorem \ref{t:4} we find further topological obstructions to parallel isometric paracomplex structure on conformally flat, scalar flat, neutral 4-manifolds:

\vspace{0.1in}
\noindent{\bf Proof of Main Theorem \ref{t:5}:}

Let $(M,g')$ be a closed, conformally flat, scalar flat, neutral 4-manifold and suppose $g'$ admits a parallel isometric paracomplex structure. Applying the Chern-Gauss-Bonnet Theorem to $g'$
\[
\chi(M)={{\frac{1}{16\pi^2}}}\int_M|Ric(g')|^2\;d^4V_{g'}.
\]
Thus if $|Ric(g')|^2\leq0$, then $\chi(M)\leq0$.

On the other hand, if from $g'$ we construct the Riemannian Einstein metric $g$ via Main Theorem \ref{t:4}, we have that
\begin{align}
\chi(M)&={{\frac{1}{32\pi^2}}}\int_M|W(g)|^2-2|Ric(g)|^2+{\textstyle{\frac{2}{3}}}S^2\;d^4V_g\nonumber\\
&={{\frac{1}{32\pi^2}}}\int_M|W(g)|^2+{\textstyle{\frac{1}{6}}}R^2\;d^4V_g\nonumber\\
&\geq0\nonumber.
\end{align}
Thus if $|Ric(g')|^2\leq0$, then we must have $\chi(M)=0$. It is known that an Einstein metric on a 4-manifold with $\chi(M)=0$ is flat \cite{hitch}.
This completes the proof.
\qed

\vspace{0.1in}

One also has:

\vspace{0.1in}

\noindent{\bf Proof of Main Theorem \ref{t:6}:}

Let $(M,g)$ be a closed Riemannian Einstein 4-manifold and suppose that  $g$ admits a parallel isometric paracomplex structure. 
Let $g'$ be the associated neutral metric of Main Theorem \ref{t:4}. Then
\[
\tau(M)={{\frac{1}{48\pi^2}}}\int_M|W^+(g')|^2-|W^-(g')|^2\;d^4V_{g'}=0,
\]
as claimed.
\qed

\vspace{0.1in}
The K3 manifold is a closed simply connected 4-manifold with a celebrated Riemannian metric that is Einstein (in fact Ricci flat) \cite{yau}. It is easily computed that $\chi(K3)=24$ and $\tau(K3)=16$ and so it satisfies the required relations (\ref{e:mats}) to admit a non-vanishing 2-plane field. Thus the K3 manifold admits an almost paracomplex structure $J$ that is isometric with respect to the Einstein metric. 

However, since $\tau(K3)\neq0$, Main Theorem \ref{t:6} means that the almost paracomplex structure $J$ cannot be parallel. Similarly,

\vspace{0.1in}

\begin{coro}
The 4-manifolds ${\mathbb{ C}}P^2\#k\overline{{\mathbb C}P}^2$ for $k=3,5,7$ admit Riemannian Einstein metrics and isometric almost paracomplex structures $J$, but $J$ cannot be parallel.
\end{coro}
\begin{proof}
Consider the closed 4-manifold obtained by the connected sum of complex projective 2-space ${\mathbb{ C}}P^2$ and $k$ copies of its complex conjugate $\overline{\mathbb CP}^2$:
\[
M={\mathbb{ C}}P^2\#k\overline{{\mathbb C}P}^2.
\]
A straightforward computation shows that
\[
\chi(M)=3+k \qquad \tau(M)=1-k,
\]
and the Hitchin-Thorpe inequality (\ref{e:htid}) implies that for $M$ to admit a Riemannian Einstein metric one must have $k\leq9$. In fact, the 4-manifolds ${\mathbb{ C}}P^2\#k\overline{\mathbb CP}^2$ for $0\leq k \leq 8$ are known to admit an Einstein metric \cite{And}.

On the other hand, the existence of a neutral metric is guaranteed on $M$ so long as the equations (\ref{e:mats}) hold, namely $k$ is odd, since
\[
\chi(M)+\tau(M)=4
\qquad{\mbox{ and }}\qquad
\chi(M)-\tau(M)=2+2k.
\]
Equivalently, the existence of a non-vanishing 2-plane field is guaranteed for $k$ odd, and by Proposition \ref{p:iso} this is enough to define an almost paracomplex structure $J$ isometric with respect to any given Riemannian metric. Thus, starting with an Einstein metric on $M$ one obtains an isometric almost paracomplex structure $J$.

However, for $k=3,5,7$ the signature of $M$ is non-zero, and so $J$ cannot be parallel, as this would contradict Main Theorem \ref{t:6}. 

\end{proof}

\end{document}